\documentclass[a4paper]{amsart}
\usepackage{amsmath,amsthm,amssymb,hyperref,mathrsfs}
\usepackage{aliascnt}
\usepackage{lmodern}
\usepackage[T1]{fontenc}

\usepackage[textsize=footnotesize]{todonotes}

\usepackage{xcolor}
\definecolor{dblue}{rgb}{0,0,0.70}
\hypersetup{
	unicode=true,
	colorlinks=true,
	citecolor=dblue,
	linkcolor=dblue,
	anchorcolor=dblue
}

\makeatletter
\expandafter\g@addto@macro\csname th@plain\endcsname{%
		\thm@notefont{\bfseries}
	}%
\expandafter\g@addto@macro\csname th@remark\endcsname{%
		\thm@headfont{\bfseries}
	}%
\makeatother

\newtheorem{theorem}{Theorem}[section]	
\newtheorem*{theorem*}{Theorem}

\newaliascnt{lemma}{theorem}
\newtheorem{lemma}[lemma]{Lemma}
\aliascntresetthe{lemma}
\newtheorem*{lemma*}{Lemma}

\newaliascnt{proposition}{theorem}
\newtheorem{proposition}[proposition]{Proposition}
\aliascntresetthe{proposition}

\newaliascnt{corollary}{theorem}
\newtheorem{corollary}[corollary]{Corollary}
\aliascntresetthe{corollary}

\theoremstyle{remark}

\newaliascnt{remark}{theorem}
\newtheorem{remark}[remark]{Remark}
\aliascntresetthe{remark}
\newaliascnt{question}{theorem}
\newtheorem{question}[question]{Question}
\aliascntresetthe{question}

\newtheorem*{question*}{Question}

\newaliascnt{definition}{theorem}
\newtheorem{definition}[definition]{Definition}
\aliascntresetthe{definition}

\newaliascnt{example}{theorem}

\aliascntresetthe{example}

\renewcommand{\restriction}{\mathbin\upharpoonright}

\newcommand{\axiom}[1]{\mathsf{#1}} 
\newcommand{\ZFC}{\axiom{ZFC}}
\newcommand{\AC}{\axiom{AC}}

\newcommand{\DC}{\axiom{DC}}
\newcommand{\ZF}{\axiom{ZF}}

\newcommand{\Ord}{\mathrm{Ord}}

\newcommand{\GCH}{\axiom{GCH}}

\newcommand{\HS}{\axiom{HS}}

\DeclareMathOperator{\cf}{cf}
\DeclareMathOperator{\dom}{dom}

\DeclareMathOperator{\sym}{sym}
\DeclareMathOperator{\stem}{stem}

\DeclareMathOperator{\id}{id}
\DeclareMathOperator{\aut}{Aut}
\DeclareMathOperator{\Col}{Col}

\DeclareMathOperator{\crit}{crit}
\DeclareMathOperator{\otp}{otp}
\DeclareMathOperator{\len}{len}
\DeclareMathOperator{\Ult}{Ult}
\newcommand{\sdiff}{\mathbin{\triangle}}

\newcommand{\forces}{\mathrel{\Vdash}}

\newcommand{\PP}{\mathbb{P}}
\newcommand{\power}{\mathcal{P}}

\newcommand{\QQ}{\mathbb{Q}}
\newcommand{\RR}{\mathbb{R}}

\newcommand{\cU}{\mathcal U}
\newcommand{\sF}{\mathscr F}

\newcommand{\sG}{\mathscr G}
\newcommand{\sH}{\mathscr H}

\newcommand{\cS}{\mathcal S}

\def\ssubset{\mathrel{\vtop{\ialign{##\crcr$\hfil\subset\hfil$\crcr\noalign{\kern1.5pt\nointerlineskip}$\hfil\widetilde{}\hfil$\crcr\noalign{\kern1.5pt}}}}\vspace*{-0.3\baselineskip}}

\def\undertilde#1{\mathord{\vtop{\ialign{##\crcr
$\hfil\displaystyle{#1}\hfil$\crcr\noalign{\kern1.5pt\nointerlineskip}
$\hfil\tilde{}\hfil$\crcr\noalign{\kern1.5pt}}}}}

\newcommand{\tup}[1]{\langle#1\rangle}

\newcommand{\middd}{\mathrel{}\middle|\mathrel{}}

\author{Yair Hayut}
\author{Asaf Karagila}
\thanks{The second author was partially supported by the Austrian Science Foundation FWF, grant I~3081-N35, and by the Royal Society Newton International Fellowship, grant no.~NF170989.}
\address[Yair Hayut]{School of Mathematical Sciences.
Tel Aviv University.
Tel Aviv 69978,
Israel}
\email[Yair Hayut]{yair.hayut@mail.huji.ac.il}
\address[Asaf Karagila]{School of Mathematics,
University of East Anglia.
Norwich, NR4~7TJ, United Kingdom
}
\email[Asaf Karagila]{karagila@math.huji.ac.il}
\date{29 May, 2019}

\title{Critical Cardinals}

\begin{document}
\begin{abstract}We introduce the notion of a \textit{critical cardinal} as the critical point of sufficiently strong elementary embedding between transitive sets. Assuming the Axiom of Choice this is equivalent to measurability, but it is well-known that Choice is necessary for the equivalence. Oddly enough, this central notion was never investigated on its own before. We prove a technical criterion for lifting elementary embeddings to symmetric extensions, and we use this to show that it is consistent relative to a supercompact cardinal that there is a critical cardinal whose successor is singular.
\end{abstract}
\maketitle
\section{Introduction}
Elementary embeddings play a central role in modern set theory. An elementary embedding is nontrivial if it is not the identity on the ordinals, and the least ordinal moved by the embedding is called the \textit{critical point}. We can prove that the critical point is a large cardinal, and requiring the target model be ``more similar to $V$'' provides us with stronger notions of cardinals. 

Assuming the Axiom of Choice we can exchange elementary embeddings for combinatorial or logical properties such as trees, ultrafilters, and compactness properties of infinitary languages . We can also prove there is an upper limit to the large cardinal hierarchy: there is no nontrivial embedding $j\colon V\to V$. However the use of the Axiom of Choice in the known proofs seems to be essential.

Without assuming the Axiom of Choice, combinatorial properties cannot provide us with elementary embeddings. And much of the choiceless work on large cardinals was in the context of obtaining combinatorial properties of large cardinals at ``accessible'' levels (e.g. $\omega_1$ carrying a measure).

In recent years, however, there is a renewed interest in large cardinals without the Axiom of Choice. Woodin's work on the HOD Conjecture and cardinals related to the existence of a nontrivial $j\colon V\to V$ are notable examples, especially since they focus on the largeness formulated in terms of embeddings.

These developments led us to ask what kind of properties we can prove on the structure of the set theoretic universe assuming there is a cardinal which is a critical point of an embedding. We present in this paper the results of this research: first by isolating the notion of a critical cardinal, then by proving a technical theorem which is used to prove that the successor of a critical cardinal can be singular.

\subsection{The structure of this paper}
The paper covers the basic technical preliminaries of supercompact Radin forcing and symmetric extensions in \autoref{section:prelims}. We define critical and weakly critical cardinals in \autoref{section:critical}, and we prove some basic positive results about them. In \autoref{section:silver} we prove a Silver-like criterion for a lifting elementary embedding from the ground model to a symmetric extension.

Finally, in \autoref{section:succ-of-crit} we prove the main theorem. It is consistent, assuming the consistency of a supercompact cardinal, that a critical cardinal's successor is singular. This utilizes the Silver-like criterion with a symmetric extension via a supercompact Radin forcing.

There are open questions throughout the paper. Of course, there are many other naturally arising questions, and we encourage the reader to come up with their own questions, as well as answers to them.
\subsection{Acknowledgements}
The work on this paper started when both authors were Ph.D.\ students of Menachem Magidor at the Hebrew University of Jerusalem, and his help and guidance were paramount for the success of this work. We would also like to thank Arthur Apter for introducing us to some questions that we failed to answer, but led us to these unexpected results. And finally, we thank the anonymous referee for their much appreciated help improving the paper and its presentation.
\section{Technical preliminaries}\label{section:prelims}
Most of our notation and terminology are standard. We use $V$ to denote the ground model in which we work, and if $M$ is a transitive class (or set) and $\alpha$ is smaller than the height of $M$, then $M_\alpha$ will denote the subset of $M$ of sets with rank $<\alpha$. We say that $A\subseteq M$ is \textit{amenable} (to $M$) if for all $\alpha$ below the height of $M$, $A\cap M_\alpha\in M$.

If $\PP$ is a notion of forcing, then it always has a maximum denoted by $1_\PP$, or $1$ when there is no chance of confusion; we will also write $q\leq p$ if $q$ is \textit{stronger} than $p$. We denote $\PP$-names by $\dot x$, and by $\check x$ the canonical $\PP$-name for $x$ in the ground model. If $\dot x$ is a $\PP$-name, we say that a condition $p$ or a name $\dot y$ \textit{appears} in $\dot x$ if there is an ordered pair $\tup{p,\dot y}\in\dot x$.

When $\{\dot x_i\mid i\in I\}$ is a class of names (possibly a proper class), we denote by $\{\dot x_i\mid i\in I\}^\bullet$ the class-name $\{\tup{1,\dot x_i}\mid i\in I\}$. Note that if $I$ is a set, then this is a $\PP$-name. We extend this notation to ordered pairs and sequences as needed. Note that using this notation $\check x=\{\check y\mid y\in x\}^\bullet$.

We use $\DC_\kappa$ to abbreviates the statement Dependent Choice for $\kappa$: Every $\kappa$-closed tree of height $\leq\kappa$ without maximal elements has a cofinal branch. $\DC$ denotes $\DC_\omega$, and $\DC_{<\kappa}$ abbreviates $\forall\lambda<\kappa,\DC_\lambda$.

Our main method for constructing models where the Axiom of Choice fails is symmetric extensions, and our forcing will be a supercompact Radin forcing. For the convenience of the reader, we have included a brief overview on both of these topics.
\subsection{Symmetric extensions}
Symmetric extensions are inner models of generic extensions, where the Axiom of Choice may fail. They are obtained by identifying a particular class of names using permutations of the forcing.

Let $\PP$ be a notion of forcing, and let $\pi$ be an automorphism of $\PP$. Then $\pi$ extends to a permutation of $\PP$-names, defined recursively\[\pi\dot x=\{\tup{\pi p,\pi\dot y}\mid\tup{p,\dot y}\in\dot x\}.\]

We say that $\sF$ is a \textit{normal filter of subgroups}\footnote{This is a terribly unfortunate overlap in the terminology of a normal filter. We will refer to normal ultrafilters in the large cardinals context as ``measures'' to avoid this confusion.} over a group $\sG$ if it is closed under finite intersections and supergroups, and for any $\pi\in\sG$ and $H\in\sF$, $\pi H\pi^{-1}\in\sF$.

If $\PP$ is a notion of forcing, $\sG$ is a group of automorphisms of $\PP$, and $\sF$ is a normal filter of subgroups over $\sG$, then we say that $\tup{\PP,\sG,\sF}$ is a \textit{symmetric system}.\footnote{We will often replace $\sF$ by a filter base which generates it, granted it satisfies the closure under conjugation.}

Let $\tup{\PP,\sG,\sF}$ be a symmetric system. We say that $\PP$-name $\dot x$ is \textit{$\sF$-symmetric} if $\sym_\sG(\dot x)=\{\pi\in\sG\mid\pi\dot x=\dot x\}\in\sF$, and $\dot x$ is \textit{hereditarily $\sF$-symmetric} if it is $\sF$-symmetric and every $\dot y$ which appears in $\dot x$ is hereditarily $\sF$-symmetric. We denote by $\HS_\sF$ the class of all hereditarily $\sF$-symmetric names. If the symmetric system is clear from the context, we omit the subscripts and write $\sym(\dot x),\HS$, etc.

The following are standard in the study of symmetric extensions. The proofs can be found, for example, in \cite{Jech2003} as Lemma~14.37 and 15.51, respectively.
\begin{lemma*}[The Symmetry Lemma]
Suppose that $\pi\in\aut(\PP)$, $\dot x$ is a $\PP$-name, and $\varphi$ is a formula in the language of set theory, then \[p\forces\varphi(\dot x)\iff\pi p\forces\varphi(\pi\dot x).\]
\end{lemma*}
\begin{theorem*}
Suppose that $\tup{\PP,\sG,\sF}$ is a symmetric system and let $G$ be a $V$-generic filter for $\PP$. Then the class $M=\HS^G=\{\dot x^G\mid\dot x\in\HS\}$ is a model of $\ZF$ satisfying $V\subseteq M\subseteq V[G]$.
\end{theorem*}
The class $M$ is called a \textit{symmetric extension}. And we can define a forcing relation $\forces^\HS$ satisfying the usual forcing theorem, and even the Symmetry Lemma for $\pi\in\sG$.

\begin{definition}
Let $\tup{\PP,\sG,\sF}$ be a symmetric system. We say that $D\subseteq\PP$ is a \textit{symmetrically dense} set if there is some $H\in\sF$ such that for all $\pi\in H$, $\pi``D=D$. We say that a filter $G\subseteq\PP$ is \textit{symmetrically $V$-generic} if it is a filter and for all symmetrically dense open sets $D\in V$, $D\cap G\neq\varnothing$.
\end{definition}

The notion of symmetrically generic filters is the one needed to interpret correctly the names in $\HS$. This is reflected in the following theorem \cite[Theorem~8.4]{Karagila2016}.
\begin{theorem*}
Let $\tup{\PP,\sG,\sF}$ be a symmetric system, and $\dot x\in\HS$. Then the following are equivalent:
\begin{enumerate}
\item $p\forces^\HS\varphi(\dot x)$.
\item For every symmetrically $V$-generic filter $G$ such that $p\in G$, $\HS^G\models\varphi(\dot x^G)$.
\item For every $V$-generic filter $G$ such that $p\in G$, $\HS^G\models\varphi(\dot x^G)$.
\end{enumerate}
\end{theorem*}

Finally, while we do not discuss iterations of symmetric extensions in full, not even in the case of a two-step iteration, it will be conceptually relevant to the Silver-like criterion for lifting elementary embeddings to symmetric extensions, so we urge the reader to glance through the second author's \cite{Karagila2016}. One definition from that context is relevant to this work, and that is the generic semi-direct product of groups.

\begin{definition}
Suppose that $\tup{\PP,\sG,\sF}$ is a symmetric system and $\dot\QQ,\dot\sH\in\HS$ such that
\begin{enumerate}
\item $\sym(\dot\QQ)=\sym(\dot\sH)=\sG$, and
\item $\forces_\PP\dot\sH$ is a group of automorphisms of $\dot\QQ$.
\end{enumerate}
The \textit{generic semi-direct product} $\sG\ast\dot\sH$ is the group of automorphisms of $\PP\ast\dot\QQ$ of the form $\tup{\pi,\dot\sigma}$,\footnote{In \cite{Karagila2016} the notation is different, but since we will not be applying these automorphisms, this is less important.} such that $\pi\in\sG$ and $\forces\dot\sigma\in\dot\sH$, and the action of $\sG\ast\dot\sH$ on $\PP\ast\dot\QQ$ is \[\tup{\pi,\dot\sigma}(p,\dot q)=\tup{\pi p,\pi(\dot\sigma\dot q)}.\]
\end{definition}
We remark that in the case of a product of two symmetric systems, considering $\sH$ and $\QQ$ with their canonical names, we actually obtain the product of $\sG\times\sH$ with the natural action on $\PP\times\QQ$.
\subsection{Supercompact Radin forcing}
Our other main technical forcing tool will be supercompact Radin forcing, we will generally follow \cite{Krueger2007} for the presentation of the forcing. For the completeness of this work, we provide the definition of the forcing and a few of its basic properties. For the rest of the discussion on supercompact Radin forcing, $\kappa$ will be a fixed supercompact cardinal.\footnote{We can generally replace ``supercompact'' by ``$\lambda$-supercompact'' for a sufficiently large $\lambda$ in all of our proofs, which of course matters for exact consistency strength, but clutters the text. We leave the calculation of the exact $\lambda$ for the interested reader.}

If $a$ is a set of ordinals, we will write $\pi_a$ as the Mostowski collapse of $a$, which is the isomorphism between $a$ and its order type. We will denote by $\power_\kappa(A)$ the set $\{a\subseteq A\mid |a|<\kappa\}$. For $x,y\in\power_\kappa(\lambda)$, we say that $x$ is a \textit{strong subset} of $y$, if $x\subseteq y$ and $|x|<y\cap\kappa$. We shall denote this by $x\ssubset y$. 

We follow Krueger \cite{Krueger2007} for the presentation of supercompact Radin forcing. We define a \textit{coherent sequence of supercompact measures on $\power_\kappa(\lambda)$} as a sequence of measures $\cU=\tup{U(\alpha,i)\mid i<o^\cU(\alpha),\alpha\leq\kappa}$ where $o^\cU\colon\kappa+1\to\Ord$ is a partial function satisfying the following properties. 
\begin{enumerate}
\item There is a function $\beta\mapsto\lambda_\beta$ defined on the the domain of $o^\cU$ such that $\lambda_\kappa=\lambda$, $\lambda_\beta\geq\beta$ is a cardinal, and if $o^\cU(\alpha)$ is defined, then $\lambda_\beta<\alpha$ for all $\beta<\alpha$. Moreover, for $\alpha$ in the domain of $o^\cU$ and $i<o^\cU(\alpha)$, \[\{a\in\power_\alpha(\lambda_\alpha)\mid\otp(a)=\lambda_{a\cap\alpha}\}\in U(\alpha,i).\]
\item Each $\tup{U(\alpha,i)\mid i<o^\cU(\alpha)}$ is a sequence of normal measures on $\power_\alpha(\lambda_\alpha)$.
\item For each $\alpha\in\dom o^\cU$ and $i<o^\cU(\alpha)$, the set $\{x\in\power_\alpha(\lambda_\alpha)\mid x\cap\alpha\in\dom o^\cU\}$ is in $U(\alpha,i)$.  \item For any $\beta< o^\cU(\alpha)$, if $j\colon V\to\Ult(V,U(\alpha,\beta))$ is the ultrapower embedding, then $j(\tup{U(\beta,i)\mid i<o^\cU(\beta),\beta<\alpha})(\alpha)=\tup{U(\alpha,i)\mid i<\beta}$.
\end{enumerate}

If $\cU$ is a coherent sequence of supercompact measures, we say that $\rho$ is a \textit{repeat point} of $\cU$ if 
\[\bigcap_{\zeta < \rho} U(\kappa, \zeta) = \bigcap_{\zeta \leq \rho} U(\kappa, \zeta),\]
or equivalently, if for every $X \in U(\kappa,\rho)$ there is $\zeta < \rho$ such that $X \in U(\kappa,\zeta)$. By counting arguments, if $o^{\cU}(\kappa) = (2^{\lambda^{<\kappa}})^+$ then $\cU$ has repeat points.

Fix $\lambda>\kappa$, and let $\Lambda\colon\kappa\to\kappa$ be a function such that some $(2^{\lambda^{<\kappa}})^+$-supercompact embedding $j$ exists with $\crit(j)=\kappa$ and $j(\Lambda)(\kappa)=\lambda$. By supercompactness of $\kappa$, there is a coherent sequence of measures $\cU=\tup{U(\alpha,i)\mid i<o^\cU(\alpha),\alpha\leq\kappa}$, where $U(\alpha,i)$ is a measure on $\power_\alpha(\Lambda(\alpha))$ (see \cite[Section~2]{Krueger2007}).

Normally the Radin forcing derived from a coherent sequence of measures, $\cU$, is denoted by $\RR(\cU)$, but since in our case $\cU$ will always be clear from context we will simply write $\RR$. Krueger also defines $A_\cU$ to be a superset of all the ``large sets'' in the Radin forcing. $A_\cU$ is the following set, \[\left\{a\in\power_\kappa(\lambda)\middd\begin{array}{l}a\cap\kappa\text{ is a cardinal,}\\ o^\cU(a\cap\kappa)\text{ is defined, and }\\ \otp(a)=\Lambda(a\cap\kappa)\end{array}\right\}.\] To simplify our notation, especially as we omit $\cU$, we will abuse the notation and write $\power_\kappa(\lambda)$ to mean $A_\cU$.

We follow Krueger's definition which appears in \cite[Section~3]{Krueger2007} with a few inconsequential modifications. A condition $p$ in $\RR$ is a finite sequence $\tup{d_0,\ldots,d_n}$ such that for all $i\leq n$,
\begin{enumerate}
\item if $i<n$
    \begin{enumerate}
	\item $d_i\in\power_\kappa(\lambda)$, or
	\item $d_i=\tup{x_i,A_i}$ such that $x_i\in\power_\kappa(\lambda)$ and $\pi_{x_i}``A_i\in\bigcap_{\beta<o^\cU(\kappa_i)}U(\kappa_i,\beta)$, where $\kappa_i=x_i\cap\kappa$,
	\end{enumerate}
\item $d_n=\tup{\lambda,A_n}$ where $A_n$ is in $\bigcap_{\alpha<o^\cU(\kappa)} U(\kappa,\alpha)$.
\item if $i<j\leq n$, then $x_i\ssubset x_j$, and if $d_j=\tup{x_j,A_j}$, then for all $a\in A_j$, $a_i\ssubset a$.
\end{enumerate}

We define $\len(p)=n$ to be the length of $p$ (in particular a condition of length $0$ has only $\tup{\lambda,A_0}$ in it).

Given two conditions $p,q\in\RR$, such that $p=\tup{d_0,\ldots,d_n}$ with $d_i=x_i$ or $d_i=\tup{x_i,A_i}$, and $q=\tup{e_0,\ldots,e_m}$ with $e_i=y_i$ or $e_i=\tup{y_i,B_i}$, we write $q\leq p$ if:
\begin{enumerate}
\item $n\leq m$,
\item $B_m\subseteq A_n$,
\item if $n>0$, then there are $i_0<\dots<i_{n-1}<m$ such that for all $k<n$,
	\begin{enumerate}
	\item if $d_k\in\power_\kappa(\lambda)$, then $d_k=e_{i_k}$,
	\item if $d_k=\tup{x_k,A_k}$, then $e_{i_k}=\tup{x_k,B_{i_k}}$, where $B_{i_k}\subseteq A_k$,
\end{enumerate}
\item for each $l<m$, such that $l$ is not any $i_k$,
	\begin{enumerate}
	\item if $n=0$, or if $l>i_n$, then either $e_l\in A_n$ or $e_l=\tup{y_l,B_l}$ with $y_l\in A_n$ and $B_l\subseteq A_n$,
	\item if $n>0$ and $k$ is the least such that $l<i_k$, then $d_k$ has the form $\tup{a_k,A_k}$ and either $e_l\in A_k$ or $e_l=\tup{y_l,B_l}$ where $b_l\in A_k$ and $B_l\subseteq A_k$.
	\end{enumerate}
\end{enumerate}
If $q\leq p$ and $m=n$, we write $q\leq^* p$ and we say that $q$ is a \textit{direct extension} of $p$.  

We denote by $\stem p$ the tuple $\tup{x_0, \dots, x_{n-1}}$. It is routine to verify that every two conditions with the same stem are compatible.

\begin{proposition}The forcing $\RR$ has several important combinatorial properties:
\begin{enumerate}
\item If $p=\tup{d_0,\ldots,d_n}$ and $\xi<\kappa$ such that whenever $d_i=\tup{x_i,A_i}$, we have that $\xi<x_i\cap\kappa$, then for any $\{p_i\mid i<\xi\}$ such that $p_i\leq^* p$ for all $i$, there is $q\leq^* p_i$ for all $i<\xi$ (\cite[Lemma~3.1]{Krueger2007}).
\item It satisfies the $(\lambda^{<\kappa})^+$-c.c. (\cite[Proposition~3.2]{Krueger2007}).
\item If $p$ is a condition and $m<\len(p)$ such that $d_m=\tup{x_m,A_m}$, then $\RR\restriction p$ is isomorphic to the product $\RR\restriction p^{<m}\times\RR\restriction p^{\geq m}$, where $p^{\geq m}$ is a condition in $\RR$, and $\RR\restriction p^{<m}$ is a condition in a Radin forcing defined on a suitable measure in the sequence $\cU$. Moreover, the isomorphism also respects direct extensions. (See \cite[Section~4]{Krueger2007} for the exact details.)
\item The Prikry property with respect to $\leq^*$ (\cite[Section~5]{Krueger2007}).
\end{enumerate}
\end{proposition}

If $G$ is a $V$-generic filter for $\RR$, then the \textit{Radin club}, $C_G$ is the generic club in $\power_\kappa(\lambda)^V$ defined as\[C_G=\{a\in\power_\kappa(\lambda)\mid\exists p\in G\ a\in\stem p\}.\]
For every $a\in C_G$, let $\kappa_a=a\cap\kappa$, then $\{\kappa_a\mid a\in C_G\}$ is a club in $\kappa$ which we will denote by $C_\kappa$.

The Radin forcing collapses all the cardinals in the interval $(\alpha,\Lambda(\alpha)^{<\alpha})$ for every $\alpha\in C_\kappa$, and depending on the length of $\cU$, $\kappa$ can remain regular, measurable, or even supercompact in the generic extension.\footnote{We again leave the subtle question of determining the exact length needed for the interested reader.}

To wrap the overview of Radin forcing, we will need to use a technique that lets us construct generic filters over inner models. The following theorem is a mild modification of an unpublished theorem of Woodin.
\begin{theorem}\label{thm:cummings-woodin}
Let $p\in\RR$ be a condition in the supercompact Radin forcing. There is an elementary embedding $k\colon V\to M$ such that in $V$ there is an $M$-generic filter $H$ for $k(\RR)$ which is compatible with $k(p)$.
\end{theorem}
Before proving the theorem, we will need the following technical lemma, which shows that $\RR$ satisfies a strong Prikry property.
\begin{lemma}\label{lemma:strong-prikry-property}
Let $D\subseteq\RR$ be a dense open set, and let $p\in\RR$ be a condition. Denote by $x_0,\ldots,x_{n-1}$ the stem of $p$. There exists a condition $p^*\leq^* p$, a finite sequence of natural numbers $m_0,\ldots,m_n$, and a finite sequence of trees $T_0,\ldots,T_n$ such that
\begin{enumerate}
\item $T_i$ is a tree of height $m_i$ of candidates for extending the $i$th member of the stem of $p^*$,
\item $T_i$ is fat, in the sense that for any $\vec t=\tup{t_0,\ldots, t_{k-1}}\in T_i$ with $k<m_i$, there is $\gamma_{\vec t}<o^\cU(x_i\cap\kappa)$ such that the possible values for extending $\vec t$ inside $T_i$ is a set in $U(x_i\cap\kappa,\gamma_{\vec t})$.
\item Any extension of $p^*$ using branches through the trees lies in $D$.
\end{enumerate}
\end{lemma}
\begin{proof}
Let $p = \tup{d_0,\dots, d_n}$ be a condition in $\RR$ and let $x_0, \dots, x_{n-1}$ be the stem of $p$. Let $x_n = \lambda$. Let us consider the name for the condition $\dot{r} \in D \cap \dot{G}$. Let $y_0, \dots, y_m$ be names for the elements in the Radin club that $r$ introduces. Let $\{i_0, \dots, i_{n}\}$ be the names for the indices, such that $y_{i_j} = x_j$. Let us assume that $r$ has the minimal length in the sense that $\dot{m}$ is minimal and $\tup{i_{j + 1} - i_j \mid j < n-1}$ is minimal in the lexicographic order.

Using the Prikry property of $\RR$ and the closure of $\leq^*$, we can find a direct extension $p^*\leq^*p$ such that for each $j\leq n$, $p^*$ decides the values $i_j$.

Note that if $d_{j+1} = x_{j+1}$ then $\forces i_{j} = i_{j+1} - 1$. 

Let $j_0$ be minimal such that $d_{j_0} = \tup{x_{j_0}, A_{j_0}}$. For each $y \in A_{j_0}$, let us pick a direct extension of the extension of $p^*$ by adding $y$ to the stem, which again decides $\dot{m}$ and $\dot{i_j}$ for all $j \leq n + 1$. Let $p^*_y$ be this direct extension. Since $p^*_y$ is compatible with $p^*$, the value of $\dot{m}$ does not change and similarly, the values of $\dot{i_j}$ only change to reflect the addition of $y$ to the stem. The only thing which is not automatically determined by $p^*$ is the membership of $y$ to the stem of $\dot{r}$ - the minimal condition in the intersection of $D$ and $\dot{G}$. The truth value of this statement is decided by $p^*_y$.

Using the closure of the measures on $x_j$ for $j > j_0$ in the stem of $p$ and the normality of the measure on $x_{j_0}$, we can construct a single condition $p^{**} \leq^* p^*$, such that for all $y \in A_{j_0}^{p^{**}}$, the extension of $p^{**}$ by appending $y$ to the stem is stronger than $p^*_y$ and in particular, it decides the membership of $y$ to $\stem \dot{r}$. 

Let $B^0_{j_0}$ be the collection of all $y\in A_{j_0}^{p^{**}}$ such that $y$ is decided to be the minimal element of $\stem \dot{r}$ which is not in $\stem p$, by the extension of the condition $p^{**}$ with $y$. Clearly, $B^0_{j_0}$ is of measure one for  $U(x_{j_0}\cap \kappa, \beta)$ for some $\beta < o^{\mathcal{U}}(x_{j_0}\cap \kappa)$ if and only if $p^{*} \forces i_{j_0} > j_0$. If $B^0_{j_0}$ is of measure zero with respect to all measures on $x_{j_0} \cap \kappa$ in the coherent sequence, then by taking an additional direct extension of $p^{**}$ we may assume that it is empty. 

Repeat the process and construct a tree of height $i_{j_0} - j_0$, such that above each element in the tree the set of successors is large with respect to one of the measures of $x_{j_0} \cap \kappa$. This tree is going to be $T_0$.

For indices $j$ above $j_0$, we essentially repeat the same argument but with a minor difference: the diagonal intersection is defined differently. Note that when considering the collection of conditions $p^*_y$ as before, we have no control over the different large sets which are attached to $x_k$ for $k < j$. Those large sets belong to measures which are not sufficiently closed, so we cannot simply intersect them. Instead, we take diagonal intersection only above $x_k$ and consider the sets 
\[B_s = \{y \in A^{p^*}_k \mid p^*_y \restriction k = s,\, p^*_y \forces y\text{ is the next element in }\stem \dot{r}\}\]
By the completeness of the measures, there is $s$ such that $B_s$ is large. Let direct extend $p^*$ by taking the diagonal intersection above coordinate $k$ and replace the lower $k$ coordinates with $s$. From this point the rest of the argument is the same.
\end{proof}
We are now ready to prove \autoref{thm:cummings-woodin}.
\begin{proof}
Let us start by arguing that it is sufficient to prove the Theorem for the case when $p$ has an empty stem. Since the forcing $\RR\restriction p$, when $p$ has a nontrivial stem, splits into a product of $\len(p)$ components, each of them being a Radin forcing using some coherent measure sequence below some condition with empty stem,\footnote{We only defined splitting for the case of $d_i=\tup{x_i,A_i}$ and not when $d_i=x_i$. To overcome this we can either treat this case as a Radin forcing with an empty sequence, which would be a trivial forcing.} by applying the theorem finitely many times over each of those components we obtain a generic filter for each component. Using Lemma \ref{lemma:strong-prikry-property}, we conclude that those filters are mutually generic and therefore their product is a generic filter for the image of $\RR$ that contains the image of $p$. So we can safely assume without loss of generality that $p$ has an empty stem.

We define by recursion a sequence of models $M_\alpha$ and elementary embeddings $j_{\alpha,\beta}\colon M_\alpha\to M_\beta$ ($M_0=V$ and $j_{\alpha,\alpha}=\id$), as well as a sequence $\vec s$ of seeds $s_\alpha$. To improve readability, we denote by $j_\alpha$ the embedding $j_{0,\alpha}$, as well as $\cU_\alpha=j_\alpha(\cU)$, $\kappa_\alpha=j_\alpha(\kappa)$ and $\lambda_\alpha=j_\alpha(\lambda)$.

Suppose that $M_\alpha$, all the embeddings up to $\alpha$, and the sequence $\vec s\restriction\alpha$ was defined. Let $\gamma<j_\alpha(o^\cU(\kappa))$ be the least ordinal such that there is no cofinal sequence $\tup{\alpha_i\mid i<\cf(\gamma)}$ such that $\cU_\alpha(\kappa_\alpha,\gamma)$ is the set 
\[\Bigl\{X\subseteq\power_{\kappa_\alpha}(\lambda_\alpha)^{M_\alpha}\Bigm|\{i<\cf(\gamma)\mid s_{\alpha_i}\in X\}\text{ contains a tail of }\cf(\gamma)\Bigr\}.\]
If no such $\gamma$ exists, then we halt. 

Let $M_{\alpha+1}$ be the ultrapower $\Ult(M_\alpha,\cU_\alpha(\kappa_\alpha,\gamma))$, and define $j_{\gamma,\alpha+1}$ as the composition of $j_{\gamma,\alpha}$ with the ultrapower embedding. In particular, $j_{\alpha,\alpha+1}$ is itself the ultrapower embedding. Finally, let $s_\alpha$ be $j_{\alpha,\alpha+1}``\lambda_\alpha$. Note that if $\alpha$ was a successor ordinal, then $\gamma=0$.

If $\alpha$ is a limit ordinal, and $M_\beta$ were defined for all $\beta<\alpha$, then $M_\alpha$ is the direct limit of these models, combined with the embeddings, and $j_{\beta,\alpha}$ is the canonical embedding from $M_\beta$ to $M_\alpha$. By Gaifman's theorem, $M_\alpha$ is well-founded.

We claim that the process halts. Otherwise, let $\mu$ be a regular cardinal larger than $2^{o^\cU(\kappa)}+2^{2^{\lambda^{<\kappa}}}$. By the definition of the direct limit of ultrapowers, each element in $M_\mu$ is of the form $j_\mu(f)(j_{\alpha_0+1,\mu}(s_{\alpha_0}),\ldots,j_{\alpha_{n-1}+1,\mu(s_{\alpha_{n-1}}}))$, where $f\colon\power_\kappa(\lambda)^n\to V$. Let $\gamma_\alpha$ denote the $\gamma$ used in the $\alpha$th ultrapower, and let $g_\alpha\colon\power_\kappa(\lambda)\to o^\cU(\kappa)$ be a function such that \[\gamma_\alpha=j_\alpha(g_\alpha)(j_{\alpha_0+1,\alpha}(s_{\alpha_0}),\ldots,j_{\alpha_{n-1}+1,\alpha}(s_{\alpha_{n-1}})).\]
Since $\alpha_0,\ldots,\alpha_{n-1}<\alpha$, there is a stationary $S\subseteq\mu$ such that all of these are the same. Similarly, since there are not many possible $g_\alpha$'s either, we may assume that for all $\alpha\in S$, $g$ and $\alpha_0,\ldots,\alpha_{n-1}$ are the same.

Let $\tup{\delta_n\mid n<\omega}\subseteq S$ such that $\delta=\sup\{\delta_n\mid n<\omega\}\in S$ as well. We claim that for all $X\in\cU_\delta(\kappa_\delta,\gamma_\delta)$ there is a natural number $N$, such that for all $n>N$, $j_{\delta_n+1,\delta}(s_{\delta_n})\in X$. Let $N$ be the least such that $X=j_{\delta_N,\delta}(Y)$ for some set $Y$. Since $\gamma_\delta=j_{\delta_N+1,\delta}(s_{\delta_N})$, it follows that $Y\in\cU_{\delta_N}(\kappa_{\delta_N},\gamma_{\delta_N})$. In particular, $s_{\delta_N}\in j_{\delta_N,\delta_N+1}(Y)$. Therefore $j_{\delta_N+1,\delta}(s_{\delta_N})\in X$, and the argument is similar for any $n>N$ as wanted. But this is a contradiction to the choice of $\gamma_\delta$, so the process had to halt at some limit ordinal.

Let $\delta$ denote the length of the recursion. We claim that $\tup{j_{\alpha+1,\delta}(s_\alpha)\mid\alpha<\delta}$ is an $M_{\delta}$-generic Radin club for $j_\delta(\RR)$. By the definition of $\delta$, for all $\gamma<j_\delta(o^\cU(\kappa))$, there are cofinally many $\alpha<\delta$ for which $j_{\alpha,\alpha+1}$ is an ultrapower embedding using a measure $U_*$ such that $j_{\alpha,\delta}(U_*)=U(\kappa_\delta,\gamma)$, thus the sequence enters any finite fat tree, and is therefore $M_\delta$-generic. Therefore taking $M=M_\delta$ and $k=j_\delta$ completes the proof.
\end{proof}
\section{Critical cardinals}\label{section:critical}
\subsection{Weakly critical cardinals}
\begin{definition}
We say that $\kappa$ is a \textit{weakly critical cardinal} if for every $A\subseteq V_\kappa$ there exists an elementary embedding $j\colon X\to M$ with critical point $\kappa$, where $X$ and $M$ are transitive and $\kappa,V_\kappa,A\in X\cap M$.
\end{definition}
Assuming Choice, this definition is equivalent to the statement that $\kappa$ is weakly compact, as a consequence of the following proposition and \cite[Theorem~4.5]{Kanamori:2003}. However, without Choice the term ``weakly compact'' is ambiguous in the sense that the many definitions need not be equivalent anymore. 

Note that by a simple coding argument we can extend the requirement from just one $A$ to $|V_\kappa|$ subsets at the same time.

\begin{proposition}
$\kappa$ is weakly critical if and only if for every $A\subseteq V_\kappa$ there is a transitive, elementary end-extension of $\tup{V_\kappa,\in,A}$.
\end{proposition}
In the standard context of weak compactness this is known as the \textit{Extension Property}, and it is due to Keisler (see also \cite[Theorem~4.5]{Kanamori:2003}).
\begin{proof}
Suppose that $\kappa$ is weakly critical, $A\subseteq V_\kappa$, and $j\colon X\to M$ is an elementary embedding witnessing the fact $\kappa$ is weakly critical. Let $W=j(V_\kappa)$ and let $B=j(A)$, then $W$ is transitive and $B\cap V_\kappa=A$. Easily, $\tup{W,\in,B}$ is an elementary end-extension of $\tup{V_\kappa,\in,A}$ as wanted.

In the other direction, suppose that $\tup{W,\in,B}$ is a transitive, elementary end-extension of $\tup{V_\kappa,\in,A}$, denote by $\kappa'=W\cap\Ord$. Let $M=W\cup\{W,\kappa',B\}$, and let $X=V_\kappa\cup\{V_\kappa,\kappa,A\}$. Then both $M$ and $X$ are transitive sets, and defining $j\colon X\to M$ by $j\restriction V_\kappa=\id$ and $j(V_\kappa)=W$, $j(\kappa)=\kappa'$ and $j(A)=B$ is an elementary embedding between transitive sets with critical point $\kappa$.
\end{proof}
\begin{proposition}\label{prop:wc-is-mahlo}
If $\kappa$ is weakly critical cardinal, then it is Mahlo, it is in particular strongly inaccessible, and in particular regular.
\end{proposition}
This is a very similar proof to the proof in $\ZFC$. It should be noted that without Choice ``strongly inaccessible'' could have different meanings (see \cite{BlassEtal:2007}), and here we mean that $V_\kappa$ satisfies second-order $\ZF$, or that $\kappa$ is a regular limit cardinal and for all $\alpha<\kappa$, $V_\alpha$ does not map onto $\kappa$.
\begin{proof}
If $\alpha<\kappa$ and there is a function $f\colon V_\alpha\to\kappa$ which is cofinal, let $j\colon X\to M$ witness that $\kappa$ is weakly critical with $f\in X$. Then $j(f)=f$, as its domain is fixed by $j$, which would have range cofinal in $j(\kappa)>\kappa$, which is a contradiction. Therefore $\kappa$ is strongly inaccessible.

Now let $A$ be the set of strongly inaccessible cardinals below $\kappa$, and let $C\subseteq\kappa$ be a club. Let $j\colon X\to M$ witness that $\kappa$ is weakly critical with $A,C\in X$. Then $M\models\kappa\in j(A)\cap j(C)$, therefore $X\models A\cap C\neq\varnothing$. Therefore $A$ is indeed stationary.
\end{proof}
The proof above generalizes to many other properties we have grown to expect from weakly compact cardinals in $\ZFC$. For example a weakly critical cardinal is Mahlo to any degree up to $\kappa^+$. Similarly, a weakly critical cardinal has the tree property.

The following question has been raised by Itay Kaplan.
\begin{question}\label{q:least-wc}
How bad can a weakly critical cardinal's identity crisis be without Choice? Can the least weakly critical cardinal be the least measurable?
\end{question}
\subsection{Critical cardinals}
\begin{definition}
We say that a cardinal $\kappa$ is a \textit{critical cardinal} if it is the critical point of an elementary embedding $j\colon V_{\kappa+1}\to M$, where $M$ is a transitive set. \end{definition}
Easily, assuming Choice a cardinal is critical if and only if it is a measurable cardinal, and we can then assume that $j$ is defined on $V$ itself by taking an ultrapower of $V$ by a normal measure. This argument uses {\L}o\'s' theorem which relies on the Axiom of Choice, so there is no reason to expect that in $\ZF$ we can replace $j$ by an embedding defined on the whole universe, $V$. Therefore, we require only an initial segment to be the domain of $j$.

Clearly, every critical cardinal is weakly critical. But even more is true.
\begin{proposition}
If $\kappa$ is critical, then $\kappa$ is carries a normal measure on $\kappa$ which concentrates on the set of weakly critical cardinals.
\end{proposition}
Again the proof is quite similar to the proof in $\ZFC$.
\begin{proof}
We define $\cU$ to be $\{A\subseteq\kappa\mid\kappa\in j(A)\}$. It is easy to check that $\cU$ is a normal measure on $\kappa$. Observe now that if $A\subseteq V_\kappa$, then $\tup{j(V_\kappa),\in,j(A)}$ is a transitive elementary end-extension of $\tup{V_\kappa,\in,A}$ inside $M$.

Therefore $M\models\kappa$ is weakly critical. Let $A$ be $\{\lambda<\kappa\mid\lambda\text{ is weakly critical}\}$, then $\kappa\in j(A)$ and therefore $A\in\cU$.
\end{proof}
It is known that $\omega_1$ can be measurable (\cite[Theorem~21.16]{Jech2003}), but it is not even a weakly critical cardinal, as follows from \autoref{prop:wc-is-mahlo}. And Eilon Bilinsky and Moti Gitik proved in \cite{BilinskyGitik2012} that it is consistent that there is a measurable cardinal with no normal measures. This shows that from a combinatorial point of view, being a critical cardinal is much stronger than being measurable.
\begin{proposition}\label{prop:ZFC-below-critical}
Let $\kappa$ be a critical cardinal. If $V_\kappa\models\ZFC$, then $V_{\kappa+1}$ can be well-ordered and $\kappa^+$ is regular and not measurable.
\end{proposition}
These consequences are somewhat similar in flavor to the theorems of Everett Bull in \cite{Bull1978}, where he proves a similar theorem for a measurable cardinal (albeit under the implicit assumption there are normal measures).
\begin{proof}
Let $j\colon V_{\kappa+1}\to M$ be an elementary embedding witnessing that $\kappa$ is critical. By elementarity $M_{j(\kappa)}\models\ZFC$ and since $V_{\kappa+1}\in M_{j(\kappa)}$, $V_{\kappa+1}$ can be well-ordered.

To see that $\kappa^+$ is regular, note that there is a definable surjection from $V_{\kappa+1}$ onto $\kappa^+$, so it follows from the well-orderability of $V_{\kappa+1}$ that there is a sequence $\tup{f_\alpha\mid\alpha<\kappa^+}$ such that $f_\alpha$ is an injective function from $\alpha$ to $\kappa$. This implies in $\ZF$ alone that $\kappa^+$ is regular and not measurable, using the standard Ulam matrix argument (see, for example, Corollary~2.4 in Kanamori's book \cite{Kanamori:2003}).
\end{proof}
Note that we do not require anything about the closure of the target model $M$, and while $\ZFC$ proves that there is always such $M$ satisfying $M^\kappa\subseteq M$, where $M$ is an ultrapower, it might not be the case in $\ZF$. This raises several questions.
\begin{question}\label{q:provable-closure}
If $\kappa$ is a critical cardinal, can we always find a transitive set $M$ such that $j\colon V_{\kappa+1}\to M$ witnesses that $\kappa$ is critical, and $M$ is closed under $\omega$-sequences? Under $<\kappa$-sequences? Under $V_\kappa$-sequences?
\end{question}
\begin{question}\label{q:ultrapowers}
If $\kappa$ is critical, and $\cU$ is a normal measure on $\kappa$, is it true that $V_{\kappa+1}^\kappa/\cU$ is well-founded and extensional? What about $V^\kappa/\cU$?
\end{question}
By Mitchell Spector's work in \cite{Spector1988}, the above question has a positive answer if and only if for every family of non-empty sets of size $\kappa$, we can find a partial choice function whose domain is in $\cU$.

We would also like to point out the obvious ways one can extend this definition by replacing $V_{\kappa+1}$ with larger initial segments, or requiring better closure properties of the target models. This leads us quite naturally to the next part.
\subsection{Supercompact cardinals}
The following formulation of supercompactness was identified by Woodin \cite[Definition~220]{Woodin2010} as the appropriate one for $\ZF$.
\begin{definition}A cardinal $\kappa$ is supercompact if for each $\alpha>\kappa$ there exist $\beta>\alpha$ and an elementary embedding, $j\colon V_\beta\to N$ such that
\begin{enumerate}
\item $N$ is a transitive set and $N^{V_\alpha}\subseteq N$,
\item $j$ has critical point $\kappa$,
\item $\alpha<j(\kappa)$.
\end{enumerate}
\end{definition}
Clearly, supercompact cardinals are critical cardinals. But the existence of arbitrarily large and arbitrarily closed target models lends to a greater impact on the structure of the universe. For example, we have the following theorem, which is also due to Woodin \cite[Lemma~225]{Woodin2010}.

\begin{theorem*}Suppose that $\kappa$ is a regular cardinal and $\DC_{<\kappa}$ holds, and $\delta$ is a supercompact cardinal such that $\delta>\kappa$. Then there is a forcing extension given by the forcing $\PP^\delta_\kappa=\Col(\kappa,<V_\delta)$ such that $\delta=\kappa^+$ and $\DC_\kappa$ holds.
\end{theorem*}

In the proof of the theorem, it is evident that $\delta^+$, and indeed all cardinals above $\delta$ are preserved. And since the extension satisfies $\DC_\kappa$, it means that $\cf(\delta^+)>\kappa$ there, which in turn implies that $\delta^+$ could not have small cofinality to begin with (we can always take $\kappa=\omega$ for this since $\DC_{<\omega}$ is a theorem of $\ZF$). We also make the observation that if $\delta$ is supercompact and $\AC$ fails, then it fails in $V_\delta$ as well. This leads to the definition of a \textit{nontrivial failure of choice} at a supercompact cardinal $\delta$, which means that $\delta$ is supercompact, but $\PP^\delta_\omega$ does not force the Axiom of Choice. In other words, the failure of choice is \textit{generated} by a specific set in $V_\delta$.

Woodin suggested, in a private communication, that the only way currently known to obtain a nontrivial failure of choice with a supercompact starts by assuming the existence of a Reinhardt cardinal, making the above question more interesting from a consistency strength point of view.
\begin{question}\label{q:nontrivial-spc}
Can we construct a model in which there is a nontrivial failure of choice above a supercompact just by starting with a single supercompact in $\ZFC$?
\end{question}
\section{Lifting embeddings to symmetric extensions}\label{section:silver}
If we want to prove theorems about critical cardinals in the absence of Choice, we need some technical machinery which allows us to create models of $\ZF$ where the Axiom of Choice is false, but there is a critical cardinal. In the case of $\ZFC$ we have Silver's theorem that lets us lift embeddings to the generic extension. Seeing how symmetric extensions are the basic tool for moving from models of $\ZFC$ to models of $\ZF+\lnot\AC$, the relevant generalization seems almost necessary.

Let $\cS=\tup{\PP,\sG,\sF}$ be a symmetric system and let $j\colon V\to M$ be an elementary embedding. We want to identify a sufficient condition for the embedding $j$ to be amenably lifted to the symmetric extension given by $\cS$. Namely, if $W$ is the symmetric extension of $V$ given by $\cS$, and $N$ is the symmetric extension of $M$ given by $j(\cS)$, we are looking for a condition that lets $j$ be extended to an embedding from $W$ to $N$, such that $j\restriction W_\alpha\in W$ for all $\alpha$. 

\begin{definition}\label{def:j-decomposable}
Let $\cS=\tup{\PP,\sG,\sF}$ be a symmetric system and let $j\colon V\to M$ be an elementary embedding. We say that $\cS$ is \textit{$j$-decomposable} if the following conditions hold.\footnote{This means that $j(\cS)$ is essentially a two-step iteration of symmetric systems with an $M$-generic for the second iterand.}
\begin{enumerate}
\item There is a condition $m\in j(\PP)$ and a name $\dot\QQ\in\HS_\sF$ such that:
\begin{enumerate}
\item $\pi\colon j(\PP)\restriction m\cong\PP\ast\dot\QQ$, $\pi$ extends the function $j(p)\mapsto\tup{p,1_\QQ}$,
\item with $\sym(\dot\QQ)=\sG$,
\item there is $\dot H\in\HS_\sF$ such that $\forces_\PP\dot H\text{ is symmetrically }\check M\text{-generic for }\dot\QQ$.
\end{enumerate} 
\item There is a name $\dot\sH$ such that: 
\begin{enumerate}
\item $\dot\sH\in\HS_\sF$ and $\forces_\PP\dot\sH\leq\aut(\dot\QQ)$, with $\sym(\dot\sH)=\sG$,
\item there is an embedding $\tau\colon\sG\to\sG\ast\dot\sH$, the generic semi-direct product, such that $\tau(\sigma)$ is given by applying $\pi$ to $j(\sigma)$,
\item $\tau(\sigma)=\tup{\sigma,\dot\rho}$, and $\dot\rho$ is a name such that $\forces_\PP\dot\rho``\dot H=\dot H$.
\end{enumerate}
\item The family $j``\sF$ is a basis for $j(\sF)$.
\end{enumerate}
\end{definition}
In the above definition, if $\pi(p^*)=\tup{p,\dot q}$ we write $\pi_0(p^*)=p$ and $\pi_1(p^*)=\dot q$.

\begin{definition}Under the notation of the previous definition and the assumption that $\cS$ is indeed $j$-decomposable, let $\dot x$ be a $j(\PP)$-name. We recursively define the \textit{partial interpretation of $\dot x$ by $\dot H$} as the $\PP$-name $\dot x^H$:\footnote{In order to be fully compatible, this would be $\dot x^{\dot H}$, but this just adds clutter to the page, so we chose to omit that dot.}
\[\dot x^H=\{\tup{\pi_0(p^*),\dot y^H}\mid\tup{p^*,\dot y}\in\dot x, \pi_0(p^*)\forces_\PP\pi_1(p^*)\in\dot H\}.\]
\end{definition}
In the rest of this section, we will use the above notation implicitly.

\begin{proposition}\label{prop:j-decomposable and j}
Suppose that $\cS$ is a $j$-decomposable symmetric system. Then for any $\sigma\in\sG$ and $(j(\PP)\restriction m)$-name, $\dot x$, we have $\sigma(\dot x^H)=(j(\sigma)\dot x)^{\sigma(H)}$.
\end{proposition}
\begin{proof}

First we do a short analysis of the interactions of various automorphisms with $j$ and $\dot H$. If $p^*\in j(\PP)\restriction m$, let $\pi(p^*)$ be $\tup{p,\dot q}$. We observe the following fact: $\pi(j(\sigma) p^*)=\tup{\sigma p,\sigma(\dot\rho(\dot q))}$, where $\tau(\sigma)=\tup{\sigma,\dot\rho}$. Suppose now that $p\forces\dot q\in\dot H$, then by 2(c) in the definition of decomposability, $p\forces\dot\rho(\dot q)\in H$ as well. Therefore, $\sigma p\forces\sigma(\dot\rho(\dot q))\in\sigma\dot H$, which is the same as saying $\pi_0(j(\sigma)p^*)\forces\pi_1(j(\sigma)p^*)\in\sigma\dot H$. 

We can now prove by induction on the rank of $\dot x$ that $\sigma(\dot x^H)=(j(\sigma)\dot x)^{\sigma H}$:
\begin{align*}
\sigma(\dot x^H) 
&= \left\{\tup{\sigma(\pi_0(p^*)),\sigma(\dot y^H)}\middd\begin{array}{l}\tup{p^*,\dot y}\in\dot x, \text{ and}\\ \pi_0(p^*)\forces\pi_1(p^*)\in\dot H\end{array}\right\}\\
&= \left\{\tup{\sigma(\pi_0(p^*)),\sigma(\dot y^H)}\middd\begin{array}{l}\tup{j(\sigma)p^*,j(\sigma)\dot y}\in j(\sigma)\dot x, \text{ and}\\ \sigma(\pi_0(p^*))\forces\sigma(\pi_1(p^*))\in\sigma(\dot H)\end{array}\right\}\\
&= \left\{\tup{\pi_0(j(\sigma)p^*),(j(\sigma)\dot y)^{\sigma(H)}}\middd\begin{array}{l}\tup{j(\sigma)p^*,j(\sigma)\dot y}\in j(\sigma)\dot x, \text{ and}\\ \sigma(\pi_0(p^*))\forces\sigma(\pi_1(p^*))\in\sigma(\dot H)\end{array}\right\}\\
&= \left\{\tup{\pi_0(j(\sigma)p^*),(j(\sigma)\dot y)^{\sigma(H)}}\middd\begin{array}{l}\tup{j(\sigma)p^*,j(\sigma)\dot y}\in j(\sigma)\dot x, \text{ and}\\ \pi_0(j(\sigma)p^*)\forces\pi_1(j(\sigma)p^*)\in\sigma(\dot H)\end{array}\right\}\\
&= (j(\sigma)\dot x)^{\sigma(H)}.\qedhere
\end{align*}
\end{proof}
Under the notation and assumptions of the previous proposition, we have these corollaries.
\begin{corollary}\label{cor:H-fixing aut and j}
If $\sigma\dot H=\dot H$, then $\sigma(j(\dot x)^H)=(j(\sigma)(j(\dot x)))^H$.\qed
\end{corollary}
\begin{corollary}\label{cor:j-HS-and-H}
If $\dot x\in j(\HS)$, then $\dot x^H\in\HS$.
\end{corollary}
\begin{proof}
Let $K$ be such that $K\subseteq\sym(\dot H)$ and $j(K)$ is a subgroup of $\sym_{j(\sG)}(\dot x)$, which exists due to the fact that $j``\sF$ is cofinal in $j(\sF)$. Then if $\sigma\in K$, by \autoref{prop:j-decomposable and j} we get that \[\sigma(\dot x^H)=(j(\sigma)\dot x)^H=\dot x^H.\] By induction on the rank of $\dot x$ the conclusion follows.
\end{proof}
\begin{corollary}\label{cor:j-and-HS}
If $\dot x\in\HS$, then $j(\dot x)^H\in\HS$.\qed
\end{corollary}
\begin{theorem}[The Basic Lifting Theorem]\label{thm:lifting-criteria}
If $j\colon V\to M$ is an elementary embedding and $\cS$ is $j$-decomposable symmetric system, then $j$ can be amenably lifted to the symmetric extension defined by $\cS$.
\end{theorem}
\begin{proof}
The class $j^*=\{\tup{\dot x,j(\dot x)^H}\mid\dot x\in\HS\}^\bullet$ is stable under all automorphisms in $\sym(\dot H)$, since if $\sigma\in\sym(\dot H)$, then \[\sigma(\tup{\dot x,j(\dot x)^H})=\tup{\sigma\dot x,(j(\sigma)j(\dot x))^H}=\tup{\sigma\dot x,(j(\sigma\dot x))^H}.\]
Therefore $j^*\restriction\HS_\alpha\in\HS$ for all $\alpha$, where $\HS_\alpha$ denotes all the names in $\HS$ with rank $<\alpha$. The elementarity of $j^*$ follows from the elementarity of $j$ and the forcing theorem.
\end{proof}
\begin{remark}
It is worth noting that $j^*$ itself might not be a class of the symmetric extension. It is unclear whether or not a subclass of $\HS$ which is stable under a large group of automorphisms is itself a class of the extension, and there is no reason to believe that it is. This is why we can only prove that the lifting is amenable, and not definable.
\end{remark}
A typical case would be where $j(\PP)\restriction m$ is particularly nice and $\sG$ fixes $\dot\QQ$ pointwise, for example, when $j(\PP)\restriction m\cong\PP\times\QQ$.

The following is a trivial generalization of the Levy--Solovay theorem.
\begin{theorem}\label{thm:Levy-Solovay}
Let $\kappa$ be a critical cardinal and let $\cS\in V_\kappa$ be a symmetric system, then $\cS$ is $j$-decomposable to any $j$ such that $\crit(j)=\kappa$.\qed
\end{theorem}
\begin{question}\label{q:closure-lifting}
\autoref{thm:lifting-criteria} tells us that an embedding can be lifted. But it tells us nothing about the closure of the target model, $M$. What sort of closure properties can we get from the amenable lifting?
\end{question}

The requirements in the definition of $j$-decomposable seem almost necessary. Of the three items in \autoref{def:j-decomposable}, item (3) is the odd duck. It seems to be somewhat limiting: if $\sF$ is $\kappa$-complete, then $j(\sF)$ is $j(\kappa)$-complete, which will often render the requirement as blatantly false. It seems to be sufficient for this proof, especially for the proof of \autoref{cor:j-HS-and-H}. This leads to these three questions.

\begin{question}\label{q:optimal-definitions}\begin{enumerate}\renewcommand{\theenumi}{\alph{enumi}}
\item What is the exact requirements needed for lifting an embedding to a symmetric extension? 
\item Moreover, since \autoref{thm:lifting-criteria} gives us that the entire embedding was lifted, what if we only want to lift an initial segment of it?
\item What can we weaken in that case? What if we only want to lift a weakly compact embedding (to obtain a weakly critical cardinal)?
\end{enumerate}
\end{question}
And of course, iterations. While the iteration of forcing extensions can be realized as a forcing extension, the question is subsumed into the Silver criterion when considering the $\ZFC$ case. This therefore raises the following question.
\begin{question}\label{q:iterations}
How do we formulate the generalization of \autoref{thm:lifting-criteria} to iterations of symmetric extensions?
\end{question}
\section{Successors of critical cardinals}\label{section:succ-of-crit}
In this section we show how very little $\ZF$ has to say about successors of a critical cardinal. We assume $\ZFC+\GCH$ for this proof.\footnote{As usual in these cases, $\GCH$ is only used to simplify cardinal arithmetic calculations and we can omit it by paying the price of a slightly less readable proof.}
\begin{theorem}\label{thm:singular-successor}
If $\kappa$ is a supercompact cardinal, then there is a symmetric extension in which $\kappa$ remains a critical cardinal, and $\kappa^+$ is singular such that $\cf(\kappa^+)<\kappa$.
\end{theorem}

\subsection{The symmetric Radin system}
Let $\lambda>\kappa$ be a limit cardinal, and let us assume that $\kappa$ is at least $(2^{\lambda^{<\kappa}})^+$-supercompact. By \cite{Krueger2007}, there is a coherent sequence over $\power_\kappa(\lambda)$, $\tilde{\cU}$ of length $(2^{\lambda^{<\kappa}})^+$. In particular, there is a repeat point, $\rho$, in $\tilde{\cU}$. Let $\cU = \tilde{\cU} \restriction \rho + 1$. Let $\RR$ be the supercompact Radin forcing which is defined from $\cU$. 

If $h\colon\lambda\to\lambda$ is a permutation, then $h_*(x)=h``x$ defines a permutation of $\power_\kappa(\lambda)$, and thus a permutation of $\power(\power_\kappa(\lambda))$ defined in a similar way: $h_{**}(A)=h_*``A$.

For every permutation $h$, the set $\{x\in\power_\kappa(\lambda)\mid h_*(x)=x\}$ is a club, so for all $A\subseteq\power_\kappa(\lambda)$, $A\sdiff h_{**}(A)$ is non-stationary, and in particular has measure zero in any normal measure on $\power_\kappa(\lambda)$. It follows that the natural action of $h_*$ and $h_{**}$ on $\RR$, defines an automorphism of $\RR$.\footnote{We are being slightly inaccurate: $h$ induces an automorphism of a dense subset of $\RR$. But we ignore this in favor of readability.} We will use $\sigma_h$ to denote this automorphism.

Let $\sG$ be the group of all automorphisms $\sigma_h$, induced by a permutation of $\lambda$, $h$, such that $h\restriction\kappa=\id$ and $h$ preserves cardinality, namely $|h(\alpha)|=|\alpha|$.

Define $F_\alpha$ to be the subgroup of $\sG$ of those $\sigma_h$ for which $h\restriction\alpha=\id$. And let $\sF$ be the filter of subgroups generated by $\{F_\alpha\mid\alpha<\lambda\text{ is a cardinal}\}$.

\begin{proposition}\label{prop:radin-normality}
For every $\sigma_h\in\sG$ and a cardinal $\alpha<\lambda$, $\sigma_h F_\alpha\sigma_h^{-1}=F_\alpha$. Consequently, $\sF$ is a normal filter of subgroups.
\end{proposition}
\begin{proof}
Note that if $\sigma_g\in F_\alpha$, then $g\restriction\alpha=\id$. By the fact that $h$ preserves cardinality, if $\xi<\alpha$, then $|h(\xi)|=|\xi|<\alpha$ and therefore, $g(h^{-1}(\xi))=h^{-1}(\xi)$.

Therefore $\sigma_h\circ\sigma_g\circ\sigma_h^{-1}=\sigma_{h\circ g\circ h^{-1}}\in F_\alpha$.
\end{proof}
Let $\cS$ denote the symmetric system $\tup{\RR,\sG,\sF}$, and let $\HS$ denote the class of hereditarily symmetric names. Since the $F_\alpha$'s generate $\sF$, we say that $F_\alpha$ is a support for $\dot x\in\HS$ if $F_\alpha$ is a subgroup of $\sym(\dot x)$.

Our goal is to show that $\cS$ is $i$-decomposable for a suitable elementary embedding $i$, and that we can control the subsets of $\lambda$ which are symmetric---in particular, we can ensure that it is not collapsed.

In many symmetric extensions we control the sets of ordinals added by using homogeneity. This is one of the main reasons why so many examples of symmetric extensions use homogeneous forcings such as Levy collapses. However, the Radin forcing is far from homogeneous. Instead, we use the following lemma to get a modicum of homogeneity which will be sufficient for our proof.

Fix $\alpha\in(\kappa,\lambda)$. For a given condition $p\in\RR$, recall that $\stem p$ is $\tup{x^p_0,\ldots,x^p_{n-1}}$. Define $\pi_\alpha(p)$ to be the sequence \[\tup{x^p_0\cap\alpha,\ldots,x^p_{n-1}\cap\alpha} ^\smallfrown \tup{\otp (x^p_i \cap \beta) \mid i < n, \beta = \cf \beta \leq \lambda}.\]
Namely, $\pi_\alpha(p)$ contains the information about the value of the projection of the stem of $p$ below $\alpha$ as well as the order types of the intersection of each $x^p_i$ with regular cardinals below $\lambda$.
\begin{lemma}\label{lemma:pseudo-homogeneity}
If $\pi_\alpha(p)=\pi_\alpha(q)$, then there is an automorphism $\sigma_h\in F_\alpha$ such that $\sigma_h(p)$ is compatible with $q$.
\end{lemma}
\begin{proof}
Recall that the members of the Radin club are strongly increasing, we define a permutation $h$ recursively in $n$ steps, such that $h_*(x_i^p)=x_i^q$ and $h\restriction\alpha=\id$. Suppose that we defined $h_{k-1}$ such that $h_{k-1}(x^p_i)=x^q_i$ for all $i<k$, with $h_{-1}=\id$. As $|x_i^p|<x^p_k\cap\kappa$ for all $i<k$, there is a bijection $g$ between $x_k^p\setminus\bigcup_{i<k}x_i^p$ and $x_k^q\setminus\bigcup_{i<k}x_i^q$. Moreover, since $x_k^p\cap\alpha=x_k^q\cap\alpha$, we can assume that $g$ does not move any point below $\alpha$ and since $|x_k^p \cap \beta| = |x_k^q \cap \beta|$ for all regular cardinal $\beta \leq \lambda$, we may assume that $g$ preserves cardinality. Therefore $h_k$ can be taken as $g\circ h_{k-1}$.
\end{proof}

Let $\bar{\lambda}$ be $|\mathrm{Reg} \cap [\kappa, \lambda]|$, i.e.\ the cardinality of the set of regular cardinals between $\kappa$ and $\lambda$.

\begin{corollary}\label{cor:sets-of-ordinals}
Let $\dot f\in\HS$ be such that $\forces\dot f\colon\check\tau\to\check\Ord$ for $\tau\in\Ord$, and suppose that $F_\alpha$ is a support for $\dot f$. Then there exists a sequence of sets $\tup{B_\rho\mid\rho<\tau}$ such that $|B_\rho|\leq\alpha^{<\kappa} \cdot \kappa^{\bar{\lambda}}$ and $\forces\forall\rho,\dot f(\rho)\in\check B_\rho$.
\end{corollary}
\begin{proof}
Let $B_\rho$ be the set $\{\xi\mid\exists p, p\forces\dot f(\check\rho)=\check\xi\}$. Then by \autoref{lemma:pseudo-homogeneity}, if $\pi_\alpha(p)=\pi_\alpha(q)$, then there is some $\sigma_h\in F_\alpha$ such that $\sigma_h(p)$ is compatible with $q$. But since $\sigma_h(\dot f)=\dot f$, $q$ and $p$ cannot force different values for $\dot f(\check\rho)$. Since there are only $\alpha^{<\kappa} \cdot \kappa^{\bar{\lambda}}$ possible values for $\pi_\alpha(p)$, it has to be the case that $|B_\rho|\leq\alpha^{<\kappa}$.
\end{proof}
\begin{corollary}
Assume that $\lambda < \kappa^{+\kappa}$. Then $\forces^\HS\check\kappa^+=\lambda$.
\end{corollary}
\begin{proof}
If $\alpha\in(\kappa,\lambda)$, then $F_\alpha$ is a support for the canonical collapse of $\alpha$, given by the Radin club. But if $f\colon\kappa\to\lambda$ is a function in the symmetric extension, it has a name with support $F_\alpha$, and the lemma tells us that $f$ can only obtain $\alpha^{<\kappa}\leq\alpha^+<\lambda$ values, and in particular it is not surjective.
\end{proof}
\begin{lemma}\label{lemma:short-filters-lifting}
If $\sF$ has a normal basis of size $<\kappa$, then there is an embedding $i\colon V\to N$ such that $\cS$ is $i$-decomposable. In particular, $\kappa$ is critical in the symmetric extension.
\end{lemma}
\begin{proof}
Recall that $\cU$ is a coherent sequence of supercompact measures on $\power_\kappa(\lambda)$ of length $\rho+1$, where $\rho$ is a repeat point of $\cU$. Let $j\colon V\to M$ be the ultrapower embedding given by $U(\kappa,\rho)$. Thus, $\crit(j)=\kappa$, $j(\kappa)>\lambda$, and $j``\lambda\in M$. 

Let $p = \tup{\lambda,A_*}\in \RR$ be any condition with an empty stem. Define $m_*$ to be $\tup{j``\lambda,j``A_*}^\smallfrown\tup{j(\lambda),j(A_*)}$, this is a condition in $j(\RR)$ which is stronger than $j(p)$. In $M$, the forcing $j(\RR)$ below the condition $m_*$ decomposes into the product $\PP \times \QQ$. Since $\rho$ is a repeat point, the map that sends a condition $p = \tup{d_0, \dots, d_n}\in \RR$ to $p' = \pi_{j``\lambda}^{-1}(p)$ is an isomorphism between $\RR$ and  $\PP$.

Inside $M$, apply \autoref{thm:cummings-woodin} to obtain some elementary $k\colon M\to N$, such that in $M$ there is an $N$-generic filter for $k(j(\RR)\restriction m_*)$ which is compatible with $m=k(m_*)$. Let $i$ denote the composition of these embeddings, $k\circ j\colon V\to N$, and let $H$ be the $N$-generic filter for $k(j(\RR))$. 

We claim that $\cS$ is $i$-decomposable.
\begin{itemize}
\item $i(\RR)\restriction m=k(\PP)\times k(\QQ)\cong\PP\times k(\QQ)\cong\RR\times k(\QQ)$. 
\item The isomorphism extends $j(p)\mapsto\tup{p,1_{k(\QQ)}}$.
\item There is an $N$-generic filter for $k(\QQ)$ in $M$, and therefore there is one in $V$.
\item If $\sigma_h\in\sG$, then $i(\sigma_h)=\sigma_{i(h)}$ acts on $k(\PP)$ pointwise, and has a remainder which is in $i(\sG)$. In particular, this allows us to decompose $i(\sigma_h)$ into $(\sigma_h,\sigma_{h'})$ where $\sigma_{h'}$ is an automorphism of $k(\QQ)$.
\item Moreover, $\sigma_{h'}``H=H$. To see that, recall $B_h=\{x\in\power_\kappa(\lambda)\mid h_*(x)=x\}$ is a club in $\power_\kappa(\lambda)$, and so it must appear in $U(\kappa,\alpha)$ for all $\alpha<o^\cU(\kappa)$. Now let $C_H=\tup{x_\eta\mid\eta<i(\kappa)}$ be the Radin club determined by $H$, then we claim that $h'_*(x_\eta)=x_\eta$ for all $\eta<i(\kappa)$. By the construction of $H$ in the proof of \autoref{thm:cummings-woodin}, $x_\eta$ is a seed for a normal measure on the image of $(i_\eta\circ j)(B_h)$, where $i_\eta$ is the $\eta$th embedding of the construction. Since $(i_\eta\circ j)(B_h)$ is a club, it is the case that $x_\eta\in i(B_h)$.
\item Finally, since $\sF$ has a small basis, its pointwise image is a basis for $j(\sF)$ by elementarity, so $j``\sF$ is a basis for $j(\sF)$.
\end{itemize}
Therefore $i$ lifts to the symmetric extension, and thus $\kappa$ is critical there.
\end{proof}
\begin{proof}[Proof of \autoref{thm:singular-successor}]
For any regular $\mu<\kappa$, take $\lambda=\kappa^{+\mu}$ and with the symmetric Radin system, force below a condition that forces the Radin club to only start above $\mu$. Then in the symmetric extension $\cf(\lambda)=\mu$. Since $\{F_\alpha\mid\alpha<\mu\}$ is a basis for $\sF$, the conditions of \autoref{lemma:short-filters-lifting} hold and $\kappa$ remains critical with a singular successor of cofinality $\mu$.
\end{proof}

The construction, however, does not fit very well with the assumption that $\cf(\kappa^+)=\kappa$ (starting with $\lambda=\kappa^{+\kappa}$), or with a more ambitious assumption that $\kappa^+$ will be measurable (starting with some $\lambda>\kappa$ which is measurable). In both cases we get that $\sup j``\lambda<j(\lambda)$, a fact which is then carried over the $i$ obtained from \autoref{thm:cummings-woodin}, and since $\sF$ is a $\lambda$-complete filter, $\bigcap i``\sF\in i(\sF)$. This means that the symmetric Radin system is not $i$-decomposable, so the embedding does not lift. Which leaves us wide open with the following questions.

\begin{question}
Suppose that $\kappa$ is a critical cardinal. Can $\kappa^+$ be of cofinality $\kappa$? Can it be measurable? Is there an embedding that can be always lifted when forcing with the symmetric Radin system for any limit cardinal $\lambda$?
\end{question}
We strongly suspect that the answer to the first question is positive, which hints that the answer to the second question could be positive as well. If, however, the answer to the first question is negative this implies that a successor of a supercompact cardinal is always regular, since its cofinality cannot be smaller than $\kappa$ itself. In either case, however, the results of this section point out the obvious gap between just having a critical cardinal to having a supercompact cardinal. We also note that Jech's Lemma about the preservation of measurability in symmetric extensions \cite[Lemma~21.17]{Jech2003} can be used to show that the symmetric Radin system would preserve the measurability of $\lambda$. So finding an embedding to be lifted is really the only missing ingredient in this construction.

Finally, an obligatory question about the structure of the cardinals above a critical cardinal which seems natural.
\begin{question}
Can a critical cardinal be the last regular cardinal?
\end{question}
\providecommand{\bysame}{\leavevmode\hbox to3em{\hrulefill}\thinspace}
\providecommand{\MR}{\relax\ifhmode\unskip\space\fi MR }
\providecommand{\MRhref}[2]{%
  \href{http://www.ams.org/mathscinet-getitem?mr=#1}{#2}
}
\providecommand{\href}[2]{#2}

\end{document}